\documentclass[12pt]{amsart}
\usepackage{amssymb}
\usepackage{graphics}
\usepackage{latexsym}
\usepackage{amsmath}
\usepackage{amssymb,amsthm,amsfonts}
\usepackage{amscd}
\usepackage[arrow, matrix, curve]{xy}
\usepackage{syntonly}
\title[Monodromy of CY threefolds]{%
    On the monodromy of the moduli space of Calabi-Yau threefolds coming from
    eight planes in $\P^3$}

\author[Ralf Gerkmann]{Ralf Gerkmann}
\address{\mbox{Universit\"at Mainz, Fachbereich 8 (Mathematik), 55099 Mainz,
    Germany}}
\email{ralf.gerkmann@math.lmu.de}
\author[Mao Sheng]{Mao Sheng${}^\dagger$}
\address{East China Normal University, Dep.\ of Mathematics, \mbox{200062
Shanghai,   P.R.\ China}} \email{msheng@math.ecnu.edu.cn,\quad\quad
sheng@uni-mainz.de}
\thanks{${}^\dagger$The second named author is supported by a postdoctoral
fellowship in the East China Normal University and is also partially
supported by the Program for Changjiang Scholars and Innovative
Research Team in University.}
\author[Duco van Straten]{Duco van Straten}
\address{\mbox{Universit\"at Mainz, Fachbereich 8 (Mathematik), 55099 Mainz,
    Germany}}
\email{straten@mathematik.uni-mainz.de}
\author[Kang Zuo]{Kang Zuo}
\address{\mbox{Universit\"at Mainz, Fachbereich 8 (Mathematik), 55099 Mainz,
Germany}}  \email{zuok@uni-mainz.de}

\begin{document}
%%%%%%%%%%%%%%%%%%%% Text italic %%%%%%%%%%%%%%%%%%%%%%%%%%%%
\theoremstyle{plain}
\newtheorem{thm}{Theorem}[section]
\newtheorem{theorem}[thm]{Theorem}
\newtheorem{lemma}[thm]{Lemma}
\newtheorem{corollary}[thm]{Corollary}
\newtheorem{proposition}[thm]{Proposition}
\newtheorem{addendum}[thm]{Addendum}
\newtheorem{variant}[thm]{Variant}
%%%%%%%%%%%%%%%%%%%% Text roman %%%%%%%%%%%%%%%%%%%%%%%%%%%%%
\theoremstyle{definition}
\newtheorem{lemma and definition}[thm]{Lemma and Definition}
\newtheorem{construction}[thm]{Construction}
\newtheorem{notations}[thm]{Notations}
\newtheorem{question}[thm]{Question}
\newtheorem{problem}[thm]{Problem}
\newtheorem{remark}[thm]{Remark}
\newtheorem{remarks}[thm]{Remarks}
\newtheorem{definition}[thm]{Definition}
\newtheorem{claim}[thm]{Claim}
\newtheorem{assumption}[thm]{Assumption}
\newtheorem{assumptions}[thm]{Assumptions}
\newtheorem{properties}[thm]{Properties}
\newtheorem{example}[thm]{Example}
\newtheorem{conjecture}[thm]{Conjecture}
\numberwithin{equation}{thm}

% Skriptbuchstaben
\newcommand{\pP}{{\mathfrak p}}
\newcommand{\sA}{{\mathcal A}}
\newcommand{\sB}{{\mathcal B}}
\newcommand{\sC}{{\mathcal C}}
\newcommand{\sD}{{\mathcal D}}
\newcommand{\sE}{{\mathcal E}}
\newcommand{\sF}{{\mathcal F}}
\newcommand{\sG}{{\mathcal G}}
\newcommand{\sH}{{\mathcal H}}
\newcommand{\sI}{{\mathcal I}}
\newcommand{\sJ}{{\mathcal J}}
\newcommand{\sK}{{\mathcal K}}
\newcommand{\sL}{{\mathcal L}}
\newcommand{\sM}{{\mathcal M}}
\newcommand{\sN}{{\mathcal N}}
\newcommand{\sO}{{\mathcal O}}
\newcommand{\sP}{{\mathcal P}}
\newcommand{\sQ}{{\mathcal Q}}
\newcommand{\sR}{{\mathcal R}}
\newcommand{\sS}{{\mathcal S}}
\newcommand{\sT}{{\mathcal T}}
\newcommand{\sU}{{\mathcal U}}
\newcommand{\sV}{{\mathcal V}}
\newcommand{\sW}{{\mathcal W}}
\newcommand{\sX}{{\mathcal X}}
\newcommand{\sY}{{\mathcal Y}}
\newcommand{\sZ}{{\mathcal Z}}
% Sonderbuchstaben mit Doppellinie
\newcommand{\A}{{\mathbb A}}
\newcommand{\B}{{\mathbb B}}
\newcommand{\C}{{\mathbb C}}
\newcommand{\D}{{\mathbb D}}
\newcommand{\E}{{\mathbb E}}
\newcommand{\F}{{\mathbb F}}
\newcommand{\G}{{\mathbb G}}
\newcommand{\HH}{{\mathbb H}}
\newcommand{\I}{{\mathbb I}}
\newcommand{\J}{{\mathbb J}}
\renewcommand{\L}{{\mathbb L}}
\newcommand{\M}{{\mathbb M}}
\newcommand{\N}{{\mathbb N}}
\renewcommand{\P}{{\mathbb P}}
\newcommand{\Q}{{\mathbb Q}}
\newcommand{\R}{{\mathbb R}}
\newcommand{\SSS}{{\mathbb S}}
\newcommand{\T}{{\mathbb T}}
\newcommand{\U}{{\mathbb U}}
\newcommand{\V}{{\mathbb V}}
\newcommand{\W}{{\mathbb W}}
\newcommand{\X}{{\mathbb X}}
\newcommand{\Y}{{\mathbb Y}}
\newcommand{\Z}{{\mathbb Z}}
\newcommand{\id}{{\rm id}}
\newcommand{\rank}{{\rm rank}}
\newcommand{\END}{{\mathbb E}{\rm nd}}
\newcommand{\End}{{\rm End}}
\newcommand{\Hom}{{\rm Hom}}
\newcommand{\Hg}{{\rm Hg}}
\newcommand{\tr}{{\rm tr}}
\newcommand{\Cor}{{\rm Cor}}
\newcommand{\GL}{\mathrm{GL}}
\newcommand{\SL}{\mathrm{SL}}
\newcommand{\Aut}{\mathrm{Aut}}
\newcommand{\Sym}{\mathrm{Sym}}
\newcommand{\ModuliCY}{\mathfrak{M}_{CY}}
\newcommand{\HyperCY}{\mathfrak{H}_{CY}}
\newcommand{\ModuliAR}{\mathfrak{M}_{AR}}
\newcommand{\Gal}{\mathrm{Gal}}
%%%%%%%%%%%%%%%%%%%%%%%%%%%%%%%%%%%%%%%%%%%%%%%%%%%%%%%%

\footnotetext[1]{This work was supported by the SFB/TR 45 ¡°Periods,
Moduli Spaces and Arithmetic of Algebraic Varieties¡± of the DFG
(German Research Foundation).}

\maketitle \centerline{{\itshape To the memory of Eckart Viehweg}}

\begin{abstract}
{\scriptsize It is a fundamental problem in geometry to decide which
moduli spaces of polarized algebraic varieties are embedded by their
period maps as Zariski open subsets of locally Hermitian symmetric
domains. In the present work we prove that the moduli space of
Calabi-Yau threefolds coming from eight planes in $\P^3$ does {\em
not} have this property. We show furthermore that the monodromy
group of a good family is Zariski dense in the corresponding
symplectic group. Moreover, we study a natural sublocus which we
call hyperelliptic locus, over which the variation of Hodge
structures is naturally isomorphic to wedge product of a variation
of Hodge structures of weight one. It turns out the hyperelliptic
locus does not extend to a Shimura subvariety of type III (Siegel
space) within the moduli space. Besides general Hodge theory,
representation theory and computational commutative algebra, one of
the proofs depends on a new result on the tensor product
decomposition of complex polarized variations of Hodge structures.}
\end{abstract}

\section{Introduction}
A fundamental result of E. Viehweg \cite{V} states that for any
polarized algebraic variety the coarse moduli space $\mathfrak{M}$
exists as a quasi-projective variety. It is of great interest to
characterize those cases in which $\mathfrak{M}$ is a locally
Hermitian symmetric variety. It has been shown in \cite{VZ1},
\cite{VZ2}, \cite{MVZ} that Arakelov-type equalities lead to
sufficient conditions for this to happen. In this paper we describe
the techniques of {\em characteristic varieties} that leads to
necessary conditions that can be checked by a straightforward
calculation in concrete examples. This leads to a computational tool
that we apply to the moduli space of double octics ramified over an arrangement of eight planes in $\P^3$.\\

As over a coarse moduli space $\mathfrak{M}$ there usually does not
exist a family, we use the following weaker notion. We say that a
proper smooth map $f:\sX \to S$ over a smooth connected base $S$ is
a {\em good family} for $\mathfrak{M}$, if the moduli map $S \to
\mathfrak{M}$ of $f$ is {\em dominant and generically finite}. The
local system $\V:=(R^n f_*\Q_{\sX})_{pr}$ of primitive cohomomogies
has the structure of a weight $n$ polarized variation of $\Q$-Hodge
structures, in short $\Q$-PVHS, over $S$. Recall that this means,
among other things, that there is an Hodge filtration
${\sF}^{\bullet}$ on the vector bundle $\sV:=\V \otimes \sO_S$ with
flat connection $\nabla$ for which {\em Griffiths transversality} $
 \nabla {\sF}^p \subset {\sF}^{p-1} \otimes \Omega^1_S
$ holds. The associated graded object $(E,\theta)=(gr_F \sV,gr_F
\nabla)=(\sum_{p+q=n} E^{p,q},\oplus_{p+q=n}\theta^{p,q})$ where
$E^{p,q}:=\sF^p /\sF^{p-1}$ are the Hodge-bundles and $\theta$ is
induced by $\nabla$ is called the associated {\em Higgs-bundle}. In
general, we call a PVHS $\V$ over $S$ of {\em Calabi-Yau type}
(CY-type) if $\rank~E^{n,0} = 1$ and the morphism of vector bundles
$$
\sT_S \longrightarrow \Hom( E^{n,0},E^{n-1,1})
$$
induced by $\theta^{n,0}: E^{n,0}\to E^{n-1,1}\otimes
\Omega^{1}_{S}$ is an isomorphism at the generic point. It follows
from the Bogomolov-Tian-Todorov theorem on the unobstructedness of
the infinitesimal deformations of a Calabi-Yau manifold $X$, that
for a good family the map $\theta^{n,0}$ is naturally identified
with the Kodaira-Spencer isomorphism, so in such a situation we
obtain a PVHS of CY-type.\\

In this paper we study a particular example of an interesting
family of Calabi-Yau threefolds. An arrangement $\mathfrak{A}$ of
eight planes in general position in $\P^3$ determines
 a double cover $X$, which is a Calabi-Yau variety with singularities
along $28$ lines. A resolution $\widetilde{X}$ of such a double
octic has $\dim H^3(\widetilde{X})=20$ and carries a weight $3$
polarized Hodge structure with Hodge numbers $(1,9,9,1)$. If we vary
the arrangement $\mathfrak{A}$ in a good family, we obtain an
irreducible weight three $\Q$-PVHS $\V$ of CY-type over a smooth $9$
dimensional base $S$. We will show several theorems about $\V$.\\

\begin{theorem}\label{MainTheorem1 introduction}
$\V$ does not factor canonically.
\end{theorem}
By this we mean the following. Associated to a Hermitian symmetric
domain $D_0=G_0/K_0$ there is a special PVHS $\W$ on $D_0$ coming
from the representation $\rho_{can}: G_0\to GL(W)$, which is called
the canonical PVHS by B. Gross. We say that $\V$ {\em factors
canonically} if the period map of $\V$ factors through the one
determined by certain canonical PVHS $(D_0,\rho_{can})$.
For a precise definition see \S3.\\

A way to exclude this happening consists of picking an appropriate
point $s \in S$, compute an appropriate {\em characteristic
subvariety} of $\V$ in $\P(\sT_{S,s})$ and compare it with the
corresponding object for $\W$. If these varieties are not isomorphic
we are done.\\

Our example is a member of a well-known infinite series of
Calabi-Yau $n$-folds coming from double covers of generic
arrangements of $2n+2$ hyperplanes in $\P^n$. For $n=1$ one has the
classical theory of four points in $\P^1$ and the associated
elliptic curves and their modulus. The paper \cite{MSY} was devoted
to $n=2$ case. Here we have K3-surfaces that are double covers of
six lines in the plane. In \cite{MSY} it was shown among other
things that in this case there is a natural good family whose
associated weight 2 PVHS factors canonically. It was asked by I.
Dolgachev (see \cite{Bo}) if the $n\geq 3$ cases are canonical with
respect to $(D^{I}_{n,n},\rho_{can})$, where $\rho_{can}:
SU(n,n)\stackrel{\wedge^n}{\longrightarrow} Sp({2n \choose n},\R)$
is the indicated representation of real Lie groups. Motivated by
this question, it has been checked in \cite{S} that the primitive
Hodge numbers of CY $n$-folds are exactly the same as predicted by
Dolgachev for all $n\in \N$.\\

In the pioneering work \cite{SYY} a result similar to Theorem
\ref{MainTheorem1 introduction} for more general moduli spaces of
configurations was given, but the methods used there are completely
different from ours. Our method is based on classical Hodge theory
(see for example \cite{Griffiths}) and can be applied to many other
concrete moduli spaces, for example the moduli spaces of the
Calabi-Yau varieties in toric varieties. Moreover, we hope to extend
the present work to the $n\geq 4$ cases.\\

Using Theorem \ref{MainTheorem1 introduction} we will prove the

\begin{theorem}\label{MainTheorem3 introduction}
Let $s \in S$ be a base point and let
$$
\tau:\pi_1(S,s) \to Sp(20,\Q)
$$
be the monodromy representation associated to $\V$. Then the image
of $\tau$ is Zariski dense in $Sp(20,\R)$.
\end{theorem}

Using results of C. Schoen and P. Deligne the above theorem implies:

\begin{corollary}
The special Mumford-Tate group of a general member in
$\mathfrak{M}_{CY}$ is $Sp(20,\Q)$.
\end{corollary}

However, there exists an interesting subvariety of
$\mathfrak{M}_{CY}$ where the special Mumford-Tate groups are proper
subgroups of $Sp(20,\Q)$. Generalizing a construction from
\cite{MSY}, we define a five-dimensional subvariety
$\mathfrak{H}_{CY} \subset \mathfrak{M}_{CY}$ that we call the {\em
hyperelliptic locus}. Over it, the Hodge structure is isomorphic to
$\wedge^3$ of a $H^1(C)$, where $C$ is a hyperelliptic curve of
genus three. It is natural to ask if this decomposition can be
extended to a larger variety $\mathfrak{H}$ that contains
$\mathfrak{H}_{CY}$. Using a calculation of characteristic
subvarieties we arrive at a negative answer.

\begin{theorem}\label{MainTheorem4 introduction}
Let $\mathfrak{H}_{CY}$ be the hyperelliptic locus of
$\mathfrak{M}_{CY}$ and $\mathfrak{H}$ be any subvariety of
$\mathfrak{M}_{CY}$ which strictly contains $\mathfrak{H}_{CY}$. Let
$f: \sX\to S$ be a good family for $\mathfrak{M}_{CY}$ whose moduli
map $S\to \mathfrak{M}_{CY}$ is dominant over $\mathfrak{H}$. Then
the restriction of $\V$ to the inverse image of $\mathfrak{H}$ does
not factor through $(D^{III}_{3},\wedge^3)$.
\end{theorem}

As a corollary we have the following
\begin{corollary}\label{Maximality of special MT group}
The special Mumford-Tate group of the Calabi-Yau threefolds in
$\mathfrak{H}_{CY}$ is a subgroup of $Sp(6,\Q)$. Furthermore
$\mathfrak{H}_{CY}$ is maximal with this property. That is, for any
irreducible subvariety $\mathfrak H$ of $\mathfrak{M}_{CY}$ which
strictly contains $\mathfrak{H}_{CY}$, the special Mumford-Tate
group of a general closed point in $\mathfrak H$ is not contained in
$Sp(6,\Q)$.
\end{corollary}

The proof of Theorem \ref{MainTheorem3 introduction} relies on
new results of Hodge-theoretical nature. There is Theorem
\ref{MainTheorem2 introduction} on the tensor product
decomposition of $\C$-PVHS, parallel to the direct sum decomposition
of $\C$-PVHS due to P. Deligne (cf. \cite{D1}). Let $\bar S$ denote
a projective manifold, $Z$ a simple divisor with normal crossing and
$S = \bar S \setminus Z$. Let $\V$ denote an irreducible $\C$-PVHS
over $S$ with quasi-unipotent local monodromy around each component
of $Z$. Fix a base point $s\in S$ let
$$
\rho : \pi_1(S,s) \longrightarrow \GL(\V_s)
$$
denote the representation of the fundamental group associated to the
underlying local system of $\C$-vector spaces, where $\V_s$ denotes
the fibre of $\V$ over $s$. Let $G$ be the Zariski closure of the
image of $\rho$ inside $\GL(\V_s)$. Assume that $G$ decomposes into
a direct product $G_1 \times \cdots \times G_k$ of simple Lie
groups. Then according to Schur's lemma, we obtain a decomposition
of local systems $\V \cong \V_1 \otimes \cdots \otimes \V_k$.

\begin{theorem}\label{MainTheorem2 introduction}
Each local system $\V_i$ admits the structure of a $\C$-PVHS such
that the induced $\C$-PVHS on the tensor product $\V_1 \otimes\cdots
\otimes  \V_k$ coincides with the given $\C$-PVHS on $\V$.
\end{theorem}

The proof of the theorem is independent of the other results in this
paper. We expect the result to be useful in other situations. In
this article it helps to prove the following classification result.

\begin{theorem}\label{MainTheorem5 introduction}
Let $S$ be a smooth quasi-projective algebraic variety and $\V$ be
an weight 3 $\Z$-PVHS over $S$ which is irreducible over $\C$ and of
quasi-unipotent local monodromies. If the Hodge numbers of $\V$ are
$(1,9,9,1)$, then after a possible finite \'{e}tale base change the
connected component of the real Zariski closure of the monodromy
group of $\V$ is one of the following:
\begin{itemize}
  \item [(A)] $SU(1,1)\times SO_{0}(2,8)$,
  \item [(B)] $SU(3,3)$,
  \item [(C)] $Sp(6,\R)$,
  \item [(D)] $Sp(20,\R)$.
\end{itemize}
\end{theorem}
Using this theorem and Theorem \ref{MainTheorem1 introduction},
\ref{MainTheorem4 introduction} to exclude cases (A), (B) and (C)
one easily obtains Theorem \ref{MainTheorem3 introduction}.

\section{Two Calabi-Yau Threefolds}

Consider an arrangement $\mathfrak{A} = (H_1,...,H_8)$ of eight planes
in $\P^3$.  Such an arrangement can be given by a matrix $A \in M(8\times 4,\C)$, the $i$-th row corresponding to the defining equation
$$
\sum_{j=1}^4 a_{ij} x_j \quad = \quad 0
$$
of the hyperplane $H_i$. We say that $\mathfrak{A}$ is {\em in
general position} if no four of the planes intersect in a point. In
terms of the matrix $A$ this means that each $(4 \times 4)$-minor is
non-zero. We now describe two closely related Calabi-Yau threefolds
associated to such an arrangement $\mathfrak{A}$ in general position.\\

\subsection{The double octic}

The planes of the arrangement $\mathfrak{A}$ determine a divisor $R
=\sum_{i=1}^8 H_i$ on $\P^3$. As the degree of $R$ is even and the
Picard group of $\P^3$ has no torsion, there exists a {\em unique}
double cover $\pi : X \rightarrow \P^3$ that ramifies over $R$. The
singular locus of such a double octic $X$ is precisely the preimage
of the singular locus of $R$. Its irreducible components are given
by the lines $H_{ij} = H_i \cap H_j$ for $1 \leq i < j \leq 8$. We
fix an ordering of the index set $I=\{(i,j) \in \N^2\;\;|\;\;1 \le i
< j \le 8\}$ and let $\phi : \widetilde{\P}^{3} \rightarrow \P^3$
denote the composition of blow-ups whose centers are the strict
transforms of $H_{ij}$, taken in the chosen order. The fibre product
$\widetilde{X} := X \times_{\P^3} \widetilde{\P}^3$ sits in a
commutative diagram
$$
\begin{array}{rcl}
\widetilde{X} & \stackrel{\psi}{\longrightarrow} & X \\
{\scriptstyle \widetilde{\pi}} \downarrow & & \downarrow {\scriptstyle \pi} \\
\widetilde{\P}^3 & \stackrel{\phi}{\longrightarrow} & \P^3
\end{array}
$$
In case we start with an arrangement in {\em general position}, the
variety $\widetilde{X}$ thus obtained is a smooth Calabi-Yau
threefold. Note however that a different ordering of $I$ yields a
different birational minimal model of the singular variety $X$.

\begin{lemma}\label{infinitesimal deformation of CY}
The space of infinitesimal deformations of $\widetilde X$ is
naturally isomorphic to the space of infinitesimal deformations of
$\mathfrak{A}$.
\end{lemma}
\begin{proof}
This follows from the description of infinitesimal deformation space
of a double covers obtained in  \cite{CVanS}. We let $\sL$ (resp.
$\tilde \sL$) be the line bundles on $\P^3$ (resp. $\widetilde
\P^3$) in the decomposition
$$
\pi_{*}\sO_X=\sO_{\P^3}\oplus \sL^{-1} \ (\textrm{resp.}\
\pi_{*}\sO_{\widetilde X}=\sO_{\widetilde \P^3}\oplus \tilde
\sL^{-1} ).
$$
It satisfies $\sL^{\otimes 2}\simeq \sO_{\P^3}(R)$ (resp. $\tilde
\sL^{\otimes 2}\simeq \sO_{\widetilde \P^3}(\widetilde R)$ for
$\widetilde R$ the strict transform of $R$ under $\phi$). One has
the decomposition of tangent sheaf
$$
\pi_{*}\sT_{\widetilde X}=\sT_{\P^3}(-\log \widetilde R)\oplus
\sT_{\widetilde \P^3}\otimes \tilde \sL^{-1}.
$$
It follows that one has natural isomorphism (see also Prop. 2.1
\cite{CVanS})
$$
H^1(\sT_{\widetilde X})\simeq H^1(\sT_{\P^3}(-\log \widetilde
R))\oplus H^1(\sT_{\widetilde \P^3}\otimes \tilde \sL^{-1}).
$$
By Corollary 4.3 in \cite{CVanS}, the space $H^1(\sT_{\P^3}(-\log
\widetilde R))$ is naturally isomorphic to the space of
\emph{equisingular deformations} of $R$ in $\P^3$. Since $R$ is a
divisor with normal crossings, this space is isomorphic to the space
of infinitesimal deformations of the arrangement $\mathfrak{A}$ in
$\P^3$. Furthermore, the space of transverse deformation
$H^1(\sT_{\widetilde \P^3}\otimes \tilde \sL^{-1})$ has dimension
(after Prop. 5.1 \cite{CVanS})
\begin{eqnarray*}
  h^1(\sT_{\widetilde \P^3}\otimes \tilde \sL^{-1}) &=& h^1(\sT_{ \P^3}\otimes \sL^{-1})+\sum_{(i,j)\in I}h^0(K_{H_{ij}}) \\
   &=& h^1(\sT_{ \P^3}\otimes \sO(-3))+\sum_{(i,j)\in I}h^0(K_{\P^1})  \\
   &=& 0,
\end{eqnarray*}
where the vanishing of the first summand follows from Bott's
vanishing theorem on homogenous vector bundles over the projective
spaces. The lemma thus follows.
\end{proof}
From now on we let $\ModuliAR$ denote the moduli space of
arrangements $\mathfrak{A}$ of eight planes in $\P^3$ in general
position, and let $\mathfrak{M}_{CY}$ denote the moduli space of
$\widetilde X$.
\begin{corollary}\label{etale of moduli spaces}
The moduli map $\mathfrak{M}_{AR} \to \mathfrak{M}_{CY}$ is
\'{e}tale.
\end{corollary}
There are many ways to construct a good family for
$\mathfrak{M}_{AR}$. Here is one. Let $\mathfrak A$ be an
arrangement in general position. The moduli point of $\mathfrak A$
in $\mathfrak{M}_{AR}$ can be uniquely represented by the matrix $A$
of the form:
$$
\left(
  \begin{array}{cccccccc}
    1 & 0 & 0 & 0  & 1 & 1&1 & 1 \\
     0 & 1 & 0 & 0  & 1 & *& * & * \\
      0 & 0 & 1 & 0 & 1 & * & * & * \\
      0 & 0 & 0 & 1 & 1 & * & * & * \\
  \end{array}
\right)^{t}
$$

Conversely a matrix $A$ in the above form whose all $4\times 4$
minors are nonzero represents an arrangement $\mathfrak A$ in general
position. Thus $\mathfrak{M}_{AR}$ can be realized as an open
subvariety of the affine space $\C^9$ and it admits a natural good
family $f_0: \sX \to {\mathfrak{M}_{AR}}$, where $\sX$ is obtained by
simultaneous resolution of the singular double octic over ${\mathfrak{M}_{AR}}$.
Note also that a good family for $\mathfrak{M}_{AR}$
gives rise to a good family for $\mathfrak{M}_{CY}$.

\begin{remark}
This construction and the above lemma can actually be generalized to
arrangements of $2n+2$ hyperplanes in $\P^n$. It yields a moduli
space of smooth CY $n$-folds whose primitive Hodge numbers are given
by $h_{pr}^{p,n-p}(\widetilde{X}) = \binom{n}{p}^2$. For details, we
refer to \cite{S}.
\end{remark}

Now we proceed with the construction of another Calabi-Yau threefold
$Y$ for a given arrangement $\mathfrak{A}$.

\subsection{The Kummer cover}\label{Kummer}

Let $A = (a_{ij})$ denote a $(8 \times 4)$-matrix associated to
$\mathfrak{A}$ as described above. Furthermore let $B = (b_{ij})$
denote a matrix in $M(4\times 8,\C)$ such that the sequence
$$
0 \longrightarrow \C^4 \stackrel{A}{\longrightarrow} \C^8
\stackrel{B}{\longrightarrow} \C^4 \longrightarrow 0
$$
is exact. We let $Y$ denote the complete intersection of the four
quadrics in $\P^7$ defined by the four equations
$$
b_{i1} y_1^2 + b_{i2} y_2^2 + ... + b_{i8} y_8^2 \quad = \quad 0
\quad,\quad 1 \leq i \leq 4.
$$
In case $\mathfrak{A}$ is in general position, the space $Y$ is smooth
(see Proposition 3.1.2 in \cite{Te0}). There is a simple relation between
the singular double octic $X$ and $Y$. To describe it,
let $G_1 = \F_2^8$ denote the elementary abelian $2$-group of order
$256$. For  $a = (a_1,...,a_8) \in G_1$ we define an automorphism
$\sigma_a : \P^7 \rightarrow \P^7$ by
$$
\sigma_a(x_1:\cdots :x_i:\cdots:x_8) =
((-1)^{a_1}x_1:\cdots:(-1)^{a_i} x_i:\cdots:(-1)^{a_8}x_8).
$$
The group $G_1$ contains a distinguished normal subgroup $N_1 \lhd G_1$
of index two, the kernel of the map $a \mapsto \sum_{i=1}^8 a_i$.

\begin{proposition}
 $X \cong Y/N_1$
\end{proposition}
\begin{proof}
First remark that the matrix $A$ defines a linear embedding $j :
\P^3 \rightarrow \P^7$ of projective spaces. As $Ker(B)=Im(A)$, the
map $\pi_1 : \P^7 \rightarrow \P^7$, $(y_1:...:y_8) \mapsto
(y_1^2:...:y_8^2)$ maps $Y$ onto the image of $j$ and realizes
$\P^3$ as the quotient of $Y$ by $G_1$. The quotient map $\pi_1$
factors over $Y/N_1$, and the degree of $Y/N_1$ over $\P^3$ is two.
The ramification locus of the resulting map $\alpha : Y/N_1
\rightarrow \P^3$ is precisely $R$. Indeed, the nontrivial element
in $G_1/N_1$ is represented by any vector $e_i = (0,...,1,...,0)$ in
the standard bases of $\F_2^8$, $1 \leq i \leq 8$. Consequently, the
ramification locus consists of all points in $\P^7$ with one
coordinate zero. The embedding $j : \P^3 \rightarrow \P^7$ maps the
planes $H_1,...,H_8$ onto the intersection of $j(\P^3)$ with the
coordinate hyperplanes in $\P^7$. As the double octic $X$ is
uniquely determined by $R$, it follows $X =Y/N_1$.
\end{proof}

This geometric relation between $Y$, $X$ and $\widetilde X$
immediately leads to the following isomorphism of Hodge-structures.

\begin{proposition}\label{Prop34}
For a given arrangement $\mathfrak{A}$ in general position there
exists a natural isomorphism
$$
\rm{H}^3(Y,\Q)^{N_1} \quad \cong \quad  \rm{H}^3(X,\Q) \quad \cong \quad\rm{H}^3(\widetilde{X},\Q)
$$
of rational polarized Hodge structures. (Here $\rm{H}^3(Y,\Q)^{N_1}$
denotes the subspace of invariants under $N_1$.)
\end{proposition}

\begin{proof}
As $X=Y/N_1$, one has immediately has $\rm{H}^3(Y,\Q)^{N_1} \quad
\cong \quad  \rm{H}^3(X,\Q)$, showing that $H^3(X,\Q)$ carries a
pure Hodge structure. Let $\psi: \widetilde{X} \longrightarrow X$ be
the resolution map described above. As in general the kernel of the
map $\psi^*: {\rm H}^3(X,\Q) \longrightarrow {\rm
H}^3(\widetilde{X},\Q)$ is the part of smaller Hodge-weight (see
Corollary 5.42 of \cite{ps}), we conclude that $\psi^*$ is
injective. As the dimensions agree, $\phi^*$ is an isomorphism.
(This, of course, can also directly be seen from the Leray spectral
sequence for $\psi$.)

\end{proof}

\subsection{The hyperelliptic locus}
There exist an interesting locus in $\ModuliAR$ where the
Hodge-structure is a third exterior power:
$$\rm{H}^3(\widetilde{X},\Q)=\wedge^3\rm{H}^1(C,\Q)$$
where $C$ is a hyperelliptic curve of genus $3$. Such a $C$ is
obtained by a two-fold cover of $\P^1$ ramified over eight points.
There exists a natural morphism
$$
\gamma : (\P^1)^3 \longrightarrow \P^3
$$
which sends the point $s = ((x_1:y_1),(x_2:y_2),(x_3:y_3))$ to
$(c_{0}(x,y):c_{1}(x,y):c_{2}(x,y):c_{3}(x,y))$, where
$x=(x_1,x_2,x_3), \ y=(y_1,y_2,y_3)$ and $c_{i}(x,y)$ is the $i$-th
coefficient of the polynomial of the variables $t,\ s$
$$
\prod_{i=1}^3 (x_i t + y_i s) = c_{0}s^3 + c_{1}s^2 t + c_{2} s t^2
+ c_{3} t^3 \in \C[s,t].
$$
This morphism $\gamma$ is Galois covering with Galois group $S^3$,
the symmetric group acting on three letters.\\

One readily checks that for any  $p = (a:b) \in\P^1$
the set $H :=\gamma(\{p\} \times \P^1 \times \P^1)$
is the hyperplane with equation
$$
b^3 z_3 - ab^2 z_2 + a^2 b z_1 - a^3 z_0 \quad = \quad 0.
$$
We say that $H$ is the plane {\em associated} to the point $p$.

\begin{lemma}\label{HyperLem1}
Let $(p_1,...,p_8)$ denote a collection of eight distinct points in
$\P^1$, and let $H_i:=\gamma(\{ p_i \} \times \P^1
\times \P^1)$ the associated planes.
Then $\mathfrak{A} = (H_1,...,H_8)$ is an arrangement of planes in
general position.
\end{lemma}
\begin{proof}
In appropriate coordinates on $\P^1$, we may assume that $p_i =(-1:a_i)$
with $a_i \in \C$ pairwise distinct ($1 \leq i \leq 8$).
The matrix $A \in M(8\times 4,\C)$ corresponding to the arrangement
$\mathfrak{A}$ then has the form
$$
A \quad = \quad \begin{pmatrix}
1 & a_1 & a_1^2 & a_1^3 \\
1 & a_2 & a_2^2 & a_2^3 \\
\vdots & \vdots & \vdots & \vdots \\
1 & a_8 & a_8^2 & a_8^3
\end{pmatrix}
$$
Every $(4 \times 4)$-submatrix of $A$ is a Vandermonde matrix, hence
its determinant is non-zero. This proves that $\mathfrak{A}$ is in
general position.
\end{proof}

Let $C$ be the hyperelliptic curve of genus $3$ branched at
$p_1,...,p_8 \in \P^1$, and let $q : C \rightarrow \P^1$ denote the
corresponding covering map.
The threefold product
$$
h : C^3 \stackrel{q^3}{\longrightarrow} (\P^1)^3
\stackrel{\gamma}{\longrightarrow} \P^3
$$
is a Galois covering of degree $2^3 \cdot 6 = 48$. Its Galois group
$G_2$ is isomorphic to a semi-direct product $N \rtimes S_3$, where
$N = \langle \iota_1,\iota_2,\iota_3 \rangle$ is the group generated
by the hyperelliptic involution. We let $N_2$ denote the index two
subgroup $N' \rtimes S_3$ of $G_2$, where $N'$ is the kernel of $N
\cong \F_2^3 \stackrel{\sum}{\longrightarrow} \F_2$, where the
isomorphism sends $\iota_1$ to $(1,0,0)$ etc. We factor $h$ over
the set of $N_2$-orbits and obtain a commutative diagram
$$
\begin{array}{rcc}
C^3 & \stackrel{\delta}{\longrightarrow} & C^3 / N_2 \\
{\scriptstyle q^3} \downarrow \hspace*{0.25cm} & & ~ \downarrow
{\scriptstyle
\pi} \\
(\P^1)^3 & \stackrel{\gamma}{\longrightarrow} & \P^3
\end{array}
$$

\begin{lemma}
The double cover $\pi : C^3/N_2 \rightarrow \P^3$ branches along the
union of the hyperplanes $H_1,...,H_8$ associated to the points
$p_1,...,p_8$. So $C^3/N_2 \cong X$, where $X$ is the double octic
determined by $H_1,...,H_8$.
\end{lemma}
\begin{proof}
The Galois group of $\pi$ is isomorphic to
$G_2/N_2$, and it is generated by the image of $\iota_1$. Hence the
ramification locus of $\pi$ is the image of the fixed locus $L_1
\subseteq C^3$ of $\iota_1$ under $\pi \circ \delta$. By the
commutativity of the above diagram, it coincides with the image of
$L_1$ under $\gamma \circ q^3$. Obviously we have
$$
q^3(L_1) \quad = \quad \{ p_1,...,p_8 \} \times \P^1 \times \P^1
$$
so that our claim follows from Lemma \ref{HyperLem1}. As $C^3/N_2$
and $X$ are both double covers of $\P^3$ with the same ramification
divisor, they are isomorphic.
\end{proof}

\begin{proposition}
Let $p_1,...,p_8$ be eight distinct points in $\P^1$, $\mathfrak{A}$
the associated arrangement of hyperplanes and $\widetilde{X}$ the CY
manifold which corresponds to this arrangement. Furthermore, let $q
: C \rightarrow \P^1$ denote the hyperelliptic curve which ramifies
at $p_1,...,p_8$. Then we have an isomorphism of rational polarized
Hodge structures
$$
\rm{H}^3(\widetilde{X},\Q) \quad \cong \quad \rm{H}^3(X,\Q) \quad \cong \quad
\wedge^3 ~\rm{H}^1(C,\Q).
$$
\end{proposition}
\begin{proof}
As  $X \cong C^3/N_2$, we have
$$
\rm{H}^3(X,\Q) \quad \cong \quad \rm{H}^3(C^3,\Q)^{N_2}.
$$
Since $S_3$ is contained in $N_2$, we have an inclusion
$$
\rm{H}^3(C^3,\Q)^{N_2} \hookrightarrow \rm{H}^3(C^3,\Q)^{S_3} \cong
\rm{H}^3(\Sym^3(C),\Q).
$$
Here $\Sym^3(C)$ denotes the symmetric threefold product of $C$,
i.e.\ the quotient space $C^3 / S_3$ where $S_3$ acts on $C^3$ by
permutation of the factors. Let $J$ denote the three dimensional
Jacobian of $C$. By the Abel-Jacobi theorem, the natural map
$\varphi : \Sym^3(C) \rightarrow J$ is birational, and it induces an
isomorphism between the middle cohomology spaces. Thus
$$
\rm{H}^3(\Sym^3(C),\Q) \cong \rm{H}^3(J,\Q) \cong \wedge^3~
\rm{H}^1(J,\Q) \cong \wedge^3~ \rm{H}^1(C,\Q).
$$
Now the latter space is $20$ dimensional, which is also the
dimension of $\rm{H}^3(\widetilde{X},\Q)$. This shows that the
inclusion $\rm{H}^3(\widetilde{X},\Q) \hookrightarrow \wedge^3~
\rm{H}^1(C,\Q)$ we constructed is actually an isomorphism.
\end{proof}

Let $\mathfrak{M}_8$ denote the moduli space of eight points in
$\P^1$, which is five dimensional. Lemma \ref{HyperLem1} shows that
there exist a natural embedding $\mathfrak{M}_8 \hookrightarrow
\mathfrak{M}_{AR}$. We denote its image by $\widetilde
{\mathfrak{H}}_{CY}$ and its image in $\mathfrak{M}_{CY}$ under the
map $\mathfrak{M}_{AR}\to \mathfrak{M}_{CY}$ by $\mathfrak{H}_{CY}$.

We call this the {\em hyperelliptic sublocus}.

\begin{remark}
It is worthwhile to remark that the  construction generalizes to
all $n \geq 2$. It produces a $2n-1$ dimensional hyperelliptic
locus in the $n^2$ dimensional moduli of CY manifolds, over which
the primitive middle dimensional rational Hodge structures are wedge
products of weight one Hodge structures. For $n = 2$ the space
$\ModuliCY$ arises from the moduli space of six lines in $\P^2$ in
general position. In \cite{MSY} the analogous sublocus $\HyperCY$
was characterized as those six lines which are tangent to a smooth
conic, and it was shown that it yields the family of Kummer
surfaces.
\end{remark}

\section{Classifying spaces and Canonical Variations}

We briefly recall some basic facts on Hodge structures and their
classifying spaces. Let $V$ denote a real vector space, $n \in \N$
and $b : V \times V \rightarrow \R$ a non-degenerate bilinear form
which is symmetric if $n$ is even and skew-symmetric if $n$ is odd.
Furthermore, let $\{ h^{p,q} \}$ denote a collection of non-negative
integers parameterized by $(p,q) \in \N_0^2$ such that
$$
h^{p,q} \neq 0 \quad \text{only if} \quad p+q = n \qquad \text{and}
\qquad h^{q,p} = h^{p,q} \quad \text{for all} \quad p,q \in \Z.
$$
The set $D$ of all real Hodge structures of type $\Phi = (V,\{
h^{p,q}\},b)$ is equipped with a natural structure of a complex
manifold, called the {\em classifying space} of Hodge structures of
type $\Phi$.  It is a homogeneous space of the form $D=G/K$, where
$G$ denotes the real Lie subgroup of $GL(V)$
consisting of all $\R$-linear automorphisms fixing $b$ and where $K$ denotes a
compact subgroup of $G$.\\

Now let $S$ be a complex manifold and fix a
base point $s \in S$. Every $\R$-PVHS $\V$ of type $\Phi$ over $S$
gives rise to a map
$$
\phi : S \rightarrow \Gamma \backslash D,
$$
called the \emph{period map} associated to $\V$. Here $\Gamma$
is the image of the monodromy representation
$$\tau: \pi_1(S,s) \rightarrow G$$
defined by $\V$ considered as a local system of real
vector spaces.
%Conversely, every period map $\phi : S \rightarrow\Gamma \backslash D$ defines an $\R$-PVHS on $S$ of the above type
We refer to Chapter I in \cite{Griffiths} for more details.\\

In practice one often encounters situations where the real structure
is lost. A typical example arises from the eigenspace decomposition
of the complexification of the real cohomology of a cyclic cover
with respect to the cyclic group action (see for example \cite{DM},
\cite{ACT}). Complex polarized variations of Hodge structures are
also natural objects in Simpson's correspondence (see \cite{Sim},
\S4). We will use this slightly generalized notion of $\C$-PVHS
in this paper and we refer to \S1 in \cite{D1} for the notions of
complex Hodge structure ($\C$-HS), complex polarized Hodge structure
($\C$-PHS) and complex polarized variations ($\C$-PVHS). In \cite{D1},
also classifying spaces of $\C$-PHS of with given Hodge numbers are
defined. The corresponding classifying spaces of $\C$-PHS are also
of form $D=G/K$ where $G$ is a real Lie group and $K$ is a compact
subgroup of $G$. The difference from that of $\R$-PHS is in that $G$
is not necessarily $Sp(2n,\R)$ or $SO(r,s)$, but can also be a
special unitary
group $SU(p,q)$.\\

{\em Locally homogenous and Canonical variations}\\

Let $D_0=G_0/K_0$ be a Hermitian symmetric domain (HSD) with $G_0$
the connected component of the automorphism group of $D_0$ and $K_0$
a maximal compact subgroup of $G_0$. There are four infinite series
of classical domains after E. Cartan:\\
\begin{itemize}
  \item [(I)] $D^{I}_{p,q}=\frac{SU(p,q)}{S(U(p)\times U(q))}, \
  p\geq q\geq 1$,
  \item [(II)] $D^{II}_{n}=\frac{SO^*(2n)}{U(n)}, \ n\geq 5$,
  \item [(III)] $D^{III}_{n}=\frac{Sp(2n,\R)}{U(n)}, \ n\geq 2$,
  \item [(IV)] $D^{IV}_{n}=\frac{SO_0(2,n)}{SO(2)\times SO(n)}, \ n\geq
  5$.
\end{itemize}
\phantom{newline}

Now let $\Gamma_0$ be a torsion free discrete subgroup of $G_0$. Let
$\rho_0: G_0\to GL(F)$ be a finite dimensional complex
representation of $G_0$. S. Zucker has shown in \S4 of \cite{Zuc},
that the complex local system $\F_{\Gamma_0}=(F \times_{\Gamma_0}
D_0)$ over the complex manifold $\Gamma_0 \backslash D_0$ admits
naturally a structure of $\C$-PVHS, which we will call, following
\cite{Zuc}, a {\em locally homogenous PVHS}. By construction, it is
clear that only when $\rho$ is defined over $\R$, $\F_{\Gamma_0}$ is
a $\R$-PVHS.\\

The following example of locally homogenous PVHS was considered by
B. Gross in \cite{Gross} and also appeared in \cite{SZ} in
connection with moduli spaces of Calabi-Yau varieties. As explained
in Proposition 1.2.6 in \cite{D1}, $D_0$ determines a special node
of the Dynkin diagram of the simple complex Lie algebra
$\mathfrak{g}_{\C}=Lie(G_0)\otimes \C$. By the standard theory on
the finite dimensional representations of semi-simple complex Lie
algebras, this node also determines a fundamental representation $W$
of $\mathfrak{g}_{\C}$, and $W$ integrates to an irreducible complex
representation $\rho_{can}$ of $G_0$. When $D_0$ is of {\em tube
type}, the representation $W$ is exactly the one considered by B.
Gross in \cite{Gross} and only in this case $W$ does admit an
$G_0$-invariant real form. The locally homogenous PVHS by the above
construction are of CY type (cf. \S1 in \cite{SZ}). Following Gross
in \cite{Gross}, we call them \emph{canonical PVHS} over $\Gamma_0
\backslash D_0$. An important property of it is that the weight of
$\W$ is equal to the rank of $D_0$.

\begin{definition}\label{geometric pseudo-modularity}
Let $(D_0,\rho_0)$ be a pair consisting of a HSD $D_0$ and a
homomorphism $\rho_0 : G_0 \rightarrow G$ of real Lie groups and let
$\tilde \psi : D_0 \rightarrow D$ the map induced from $\rho_0$ and
$$\psi: \Gamma_0\backslash D_0 \to \Gamma\backslash D, \;\;\;\Gamma_0:=\rho^{-1}_0(\Gamma)$$

1) We say that a $\C$-PVHS $\V$  over $S$ {\em factors} with respect
to $(D_0,\rho_0)$, when the period map $\phi:S \to \Gamma\backslash
D$ of $\V$ factors over $\psi$. That is, we have a diagram
$$
\xymatrix{
  S \ar[rr]^{\phi} \ar[dr]_{j}
                &  &   \Gamma \backslash D      \\
                & \Gamma_0\backslash D_0 \ar[ur]_{\psi}                  }
$$
commutes.

2) For a $\C$-PVHS $\V$ of CY type we say that $\V$ factors
\emph{canonically} if it factors with respect to certain
$(D_0,\rho_{can})$
\end{definition}

\begin{lemma}\label{modularity on the level of monodromy rep}
If $\V$ factors with respect to $(D_0,\rho_0)$, then $j$ induces an
isomorphism of $\C$-PVHS $\V\simeq j^*\F$, where $\F$ is the locally
homogenous PVHS induced by $\rho_0$. Moreover the monodromy
representation $\tau : \pi_1(S,s) \rightarrow G$ of $\V$ factorizes
over $\rho_0$, i.e. there exists a homomorphism $\tau_0: \pi_1(S,s)
\rightarrow G_0$ such that the diagram
$$
\xymatrix{
  \pi_{1}(S,s) \ar[rr]^{\tau} \ar[dr]_{\tau_0}
                &  &   G      \\
                & G_{0} \ar[ur]_{\rho_0}                  }
$$
commutes.
\end{lemma}

\begin{proposition}\label{ModularTypeA}
Let $\V$ be the $\C$-PVHS associated with a good family for the
coarse moduli space $\mathfrak{M}_{CY}$. If $\V$ factors
canonically, then it must factor with respect to either
$(D^{I}_{3,3},\wedge^3)$ or $(D^{I}_{1,1}\times D^{IV}_{8},
\textrm{id}\otimes \textrm{id})$.
\end{proposition}
\begin{proof}
Assume $\V$ factors with respect to $(D_0,\rho_{can})$ for certain
HSD $D_0$. By the local Torelli theorem for Calabi-Yau manifolds we
know that $D_0$ is at least $9$ dimensional. We write $D_0=D_1\times
\cdots\times D_k$ be the decomposition into product of irreducible
HSD's. Then $\rho_{can}=\rho_{can,1}\otimes \cdots\otimes
\rho_{can,k}$ with $\rho_{can,i}$ canonical PVHS over $D_i$. By
Schur's lemma $\V \simeq \V_{1}\otimes\cdots\otimes \V_{k}$ decomposes accordingly. Since the weight of $\V$ is 3, then $k\leq 3$. \\
{\bf Case $k=1$:}. Then $D_0$ must be of tube type and $\W$
associated with $\rho_{can,1}$ must be the canonical $\R$-PVHS. By
the classification in \cite{Gross}, the pair in the
statement is the only possibility with correct Hodge numbers.\\
{\bf Case $k=2$:}. Since $D_0$ has rank 3, we can assume $D_1$ has
rank 1 and hence $D_1$ is the unit disk. It follows also the Hodge
numbers of $\V_1$ over $D_1$ is $1,1$. Since $D_2$ supports
canonical $\R$-PVHS, it must be a type IV domain. After checking the
Hodge numbers, one sees immediately the pair $(D^{I}_{1,1}\times
D^{IV}_{8}, \textrm{id}\otimes \textrm{id})$ is the unique possibility.\\
{\bf Case $k=3$:}. It follows that each $D_i,\ i=1,2,3$ has rank 1
and so the dimension of $D_1\times D_2\times D_3$ is less than 9.
\end{proof}

\section{Characteristic Subvariety}
\label{Plethysm}

We refer the reader to \S3 in \cite{Griffiths1} and references
therein for an account of the theory of {\em infinitesimal
variations of Hodge structures}, in short IVHS, initiated by P.
Griffiths. There is an important series of invariants of IVHS of a
$\C$-PVHS $\V$ of CY type over $S$, namely the {\em characteristic
subvarieties} which are contained in the projectivized tangent
bundle $\P(\sT_{S})$. The basic theory of characteristic
subvarieties is developed in \cite{SZ}. We recall the definition.

\begin{definition}
Let $\V$ a $\C$-PVHS of weight $n$ over $S$ of
weight $n$ and $(E,\theta)$ the associated Higgs-bundle.
For every $k$ with $1 \leq k \leq n$, the $k$-th
iterated Higgs field defines a morphism
$$
\theta^k : \Sym^k(\sT_S) \longrightarrow \Hom(E^{n,0},E^{n-k,k})
$$
with dual map $(\theta^{k})^* : \Hom(E^{n,0},E^{n-k,k})^*
\rightarrow \Sym^k(\Omega_S)$. Let $\mathfrak{a}_k$ denote the ideal
generated by the image of $(\theta^{k})^*$ in the symmetric algebra
$\Sym^\bullet(\Omega_S)$. Then for every $k$ with $1 \leq k \leq
n-1$, the projective variety
$$
\sC_k \quad = \quad \rm{Proj} (\Sym^\bullet(\Omega^1_S) /
\mathfrak{a}_{k+1} ) \subset \P(\sT_S)
$$
over $S$ is called the {\em $k$-th characteristic subvariety} of
$\V$.
\end{definition}
\begin{remark}
For a proper smooth family $f:\sX\to S$ of CY $n$-folds the Yukawa
coupling of $f$ is the section of $\Sym^n(\Omega^1_{S})\otimes
(R^nf_{*}\sO_{\sX})^{\otimes 2}$ defined by the $n$-th iterated
Kodaira-Spencer maps of $f$. It has significance in physics and is
an important invariant in the study of geometry on moduli spaces of
CY manifolds. Note that the $(n-1)^{th}$ characteristic subvariety
is just the vanishing locus of the Yukawa coupling.
\end{remark}

Let $s\in S$ and $(\sC_{k})_s$ be the fiber of
$\sC_k$ over $s$, which is a subvariety of $\P(\sT_{S,s})$.
following simple lemma characterizes the tangent vectors at $s$
whose classes lie in the reduced subvariety $(\sC_{k})_s^{red}$.

\begin{lemma}[Lemma 3.2 in \cite{SZ}]\label{key lemma}
Let $v\in \sT_{S,s}$ be a non-zero tangent vector, $[v]$ its class
$\P(\sT_{S,s})$ and   and $v^{k+1} \in \Sym^{k+1}(\sT_{S,s})$ the
$k+1$-th symmetric tensor power of $v$. Then:

\centerline{$[v]\in (\sC_{k})_s^{red} \subset \P(\sT_{S,s})$
if and only if $v^{k+1}\in \ker(\theta^{k+1})$.}
\end{lemma}

\begin{corollary}\label{char variety detects the ramification locus }
Let $\mathfrak{M}$ be a coarse moduli space of polarized smooth CY
$n$-folds and $f: \sX\to S$ be a good family for it. Let $s\in S$ be
a point in the ramification locus of the moduli map $S\to
\mathfrak{M}$ of $f$. Then there is a projective linear subspace in
$(\sC_{k})_s^{red}$ for $1\leq k\leq n-1$.
\end{corollary}
\begin{proof}
The question is analytically local. Let $X$ be the fiber of $f$ over
$s$ and $[s]$ be the image of $s$ in $\mathfrak M$. By the
Bogomolov-Tian-Todorov theorem the differential of the moduli map of
$f$ at $s$ is naturally identified with the Kodaira-Spencer map:
$$
\rho_{f,s}: \sT_{S,s}\to \rm{H}^{1}(X,\sT_{X}).
$$
Let $v\in \sT_{S,s}$ be a nonzero tangent vector and $\omega$ be a
generator of $\rm{H}^0(X,K_X)$. By Griffiths we have the formula for
the Higgs field action:
$$
\theta_v(\omega)=\rho_{f,s}(v)\cup \omega,
$$
where the cup product induces isomorphism
$\rm{H}^{1}(X,\sT_{X})\otimes \rm{H}^{0}(X,K_{X})\simeq
\rm{H}^{1}(X,\Omega^{n-1}_{X})$ for the CY manifold $X$. So it is
clear that the kernel of $\rho_{f,s}$ is exactly the kernel of the
Higgs field $\theta$ at $s$. In particular the corollary follows
from Lemma \ref{key lemma}.
\end{proof}

The main result in \cite{SZ} identifies the characteristic
subvarieties of the canonical PVHS over an irreducible HSD with the
characteristic bundles introduced by N. Mok in \cite{Mok}.
\begin{theorem}[Theorem 3 in \cite{SZ}]\label{identification}
Let $D_0=G_0/K_0$ be an irreducible HSD of rank $n$, $\Gamma_0$ be a
torsion free discrete subgroup of $G_0$ and let $(E,\theta)$ be the
system of Hodge bundles associated to the canonical PVHS over
$T:=\Gamma_0 \backslash D_0$. Then for each $k$ with $1\leq k\leq
n-1$ the $k$-th characteristic subvariety $\sC_{k}$ of $(E,\theta)$
over $T$ coincides with the $k$-th characteristic bundle $\sS_{k}$
over $T$.
\end{theorem}

The characteristic bundles are first defined over $D_0$ and
invariant under the $G_0$-action. By taking the quotient under
the group $\Gamma$ one obtains the characteristic bundles on $T$.
For each irreducible HSD the characteristic bundles are explicitly
described in Appendix (III.3) in \cite{M1}. For the purpose of this
article we need only the information of the first characteristic
bundle for the HSDs $D^{I}_{3,3}, \ D^{III}_{3},\ D^{IV}_8$. For the
convenience of the reader, we state them explicitly here. The
following proposition is direct consequence of Theorem
\ref{identification}.

\begin{proposition}\label{irreducible case}
Notations as Theorem \ref{identification} and $t \in T$ arbitrary
point. Then $(\sC_1)_t$ as subvariety of $\P(\sT_{T,t})$ is
isomorphic to
\begin{itemize}
  \item [(i).] $\mathbb{P}^2 \times \mathbb{P}^2\hookrightarrow \P^8$ with
  the Segre embedding, when $D_0=D^{I}_{3,3}$,
  \item [(ii).] $\mathbb{P}^2 \hookrightarrow \P^5$ with
  the Veronese embedding, when $D_0=D^{III}_{3}$,
  \item [(iii).] $Q \hookrightarrow \P^7$ with
  $Q$ a smooth quadratic hypersurface, when $D_0=D^{IV}_{8}$.
\end{itemize}
\end{proposition}

For the application we also need to work out the first
characteristic subvariety in the reducible case
$D_0=D^{I}_{1,1}\times D^{IV}_{8}$.
\begin{proposition}\label{reducible case}
Notations as last proposition. Let $D_0$ be the HSD
$D^{I}_{1,1}\times D^{IV}_{8}$. Then for each point of $t \in T$,
$(\sC_1)_t$ as subvariety of $\P(\sT_{T,t})$ is isomorphic to a
disjoint union of a point $P$ with a smooth quadratic hypersurface
$Q$ in a hyperplane of $\P^8$ away from $P$. In particular, it is
not equidimensional.
\end{proposition}
\begin{proof}
We put $D_1=D^{I}_{1,1}$ and $D_2=D^{IV}_{8}$ and fix the base point
$0\in D_0$. Then the canonical PVHS $\W$ over $D_0$ is given by
tensor product $\W_1\otimes \W_2$ of the canonical PVHS $\W_i$ over
$D_i,i=1,2$. Let $(E_i,\theta_i)$ be the corresponding Higgs-bundle
to $\W_i$. Then the corresponding Higgs-bundle $(E,\theta)$ to $\W$
is given by $(E_1\otimes E_2, \theta_1\otimes id+id\otimes
\theta_2)$. Let $p_i: D_0\to D_i,\ i=1,2$ be the natural projection.
Then one has natural isomorphism $\sT_{D_0}\simeq
p_1^*(\sT_{D_1})\oplus p_2^*(\sT_{D_2})$. Under this isomorphism, we
represent a tangent vector $v \in \sT_{D_0,0}$ by a pair $(v_1,v_2)$
with $v_i\in \sT_{D_i,0}$. Take a nonzero vector $e=e_1\otimes
e_2\in (E^{1,0}_1\otimes E^{2,0}_2)_{0}$. Then by Lemma \ref{key
lemma}, the fiber over $0$ of the first characteristic subvariety of
$D_0$ is determined by
$$
\{[v]\in \P(\sT_{D_0,0})|(\theta_{v})^2(e)=0\}.
$$
A simple calculation shows that
$$
(\theta_{v})^2(e)=2((\theta_1)_{v_1}(e_1)\otimes
(\theta_2)_{v_2}(e_2))+e_1\otimes ((\theta_2)_{v_2})^2(e_2).
$$
Hence $v$ is a characteristic vector if and only if $v_2=0$ or,
$v_1=0$ and $v_2$ is a characteristic vector in $\sT_{D_2,0}$. Thus
the fiber over 0 of the first characteristic subvariety is the
disjoint union of a point with the fiber over 0 of the first
characteristic subvariety of $D_2$, which is isomorphic to a smooth
quadratic hypersurface in $\P^7$ by Proposition \ref{irreducible case} (iii).
\end{proof}

\begin{corollary}\label{CharSubMainCorollary}
If $\V$ factors canonically, then for any $s\in S$ away from the
ramification locus of the moduli map $S\to \mathfrak{M}_{CY}$, then
$(\sC_1)_s \subset \P(\sT_{S,s})$ is isomorphic to
 $\P^2 \times \P^2\subset \P^8$ or the $P\cup Q\subset \P^8$ as in
Proposition \ref{reducible case}.
\end{corollary}
\begin{proof} This is a consequence of Proposition
\ref{ModularTypeA} and Propositions \ref{irreducible case},
\ref{reducible case}.
\end{proof}

\section{Explicit Infinitesimal Variation of Hodge Structures of the Calabi-Yau Threefolds}
\label{ExplicitHodge}

In those cases where there is an explicit description of the
cohomology with the help of an jacobian ring, it is usually also
possible to construct the infinitesimal invariants of the
corresponding variation. In particular, it will be possible to
compute the characteristic subvarieties. We will illustrate this for
the local system $\V$ for a good family of double octics $f: \sX \to
S$ with fibres $\widetilde X$ covered by $Y$.

\subsection{Jacobian Rings}

As $Y$ is a complete intersection, one can find a description of the
cohomology in terms of a certain Jacobian ring with the aid of the
{\em Cayley-trick}. Let $$R_{8,4} = \C[x_1,...,x_8,y_1,...,y_4]$$
denote the polynomial ring over $\C$ in $12$ variables. Let the four
quadrics defining $Y$ are given by
$$
f_i \quad = \quad \sum_{j=1}^8 b_{ij} x_j^2 \quad,
$$
and define $$F = \sum_{i=1}^4 y_i f_i\in R_{8,4}.$$ Let
$\mathfrak{J}_{\mathfrak{A}}$ denote the homogenous ideal generated
by the twelve partial derivatives $\frac{\partial F}{\partial x_j}$
and $\frac{\partial F}{\partial y_i}$, and put
$$
R_{Y} =R_{8,4} / \mathfrak{J}_{\mathfrak{A}}.$$ There is a natural
bigrading on $R_{8,4}$ which assigns the value $(0,1)$ to each of
the variables $x_i$ and $(1,-2)$ to each $y_i$. It induces a
bigrading on $R_{Y}$. For $0 \leq p \leq 3$ we let $R_{Y}^{(p)}$
denote the subspace generated by monomials of bidegree $(p,0)$, that
is, monomials whose total degree in the $x_j$ is $2p$ and whose
total degree in the $y_i$ is $p$.\\

\begin{proposition}\label{CompThm1}
There is an isomorphism $$\rm{H}^3(Y,\C) \cong R_{Y}$$ of
$\C$-vector spaces which identifies
$$\rm{H}^{3-p,p}(Y) \cong R_{Y}^{(p)},\;\;\;0 \leq p \leq 3$$
\end{proposition}
\begin{proof}
See \cite{Nagel}, Prop.\ 2.2.10 on page 40 or \cite{Te}.
\end{proof}

In \S \ref{Kummer} we defined the group $G_1$ of sign-changes and
its subgroup $N_1$ of index two. These groups naturally act on the
polynomial ring $R_{8,4}$ by sign-changes on the  $x_j,1 \leq j\leq
8$ and trivially on the $y_i, 1 \leq i\leq 4$. There is an induced
action of $N_1$ on the quotient $R_{Y}$. We let $\widetilde{R}_{Y}$
denote the subring $R_{Y}^{N_1}$ of elements in $R_{Y}$ fixed by
every $\sigma \in N_1$.

\begin{corollary}

It induces an isomorphism between the subspaces
$$
\rm{H}^3(\widetilde{X},\C) \quad \cong \quad \widetilde{R}_{Y}
$$
which is also compatible with the Hodge decomposition and the total
grading.
\end{corollary}
\begin{proof}
By Proposition 3.4, $\rm{H}^3(\widetilde X,\Q)$ is the subspace of
$\rm{H}^3(Y,\Q)$ invariant under $N_1$-action and moreover is a sub
PHS of $\rm{H}^3(Y,\Q)$. So one has
$$
\rm{F}^{p}\rm{H}^3(\widetilde X,\C)=\rm{F}^{p}\rm{H}^3(Y,\C)\cap
\rm{H}^3(\widetilde X,\C)=(\rm{F}^{p}\rm{H}^3(Y,\C))^{N_1}.
$$
So the assertion follows.
\end{proof}

\subsection{Multiplication and Cup-Product}
\begin{proposition}\label{PropGM}
For $0 \leq p \leq 2$ the following diagram commutes
$$
\begin{array}{ccc}
\sT_{S,s} \otimes \rm{H}^{3-p,p}(\widetilde{X}) &
\stackrel{\theta_s}{\longrightarrow}
& \rm{H}^{2-p,p+1}(\widetilde{X}) \\[0.2cm]
{\scriptstyle\cong} \downarrow & & \downarrow {\scriptstyle\cong}
\\[0.2cm]
\widetilde{R}^{(1)}_{Y} \otimes \widetilde{R}^{(p)}_{Y}
 & \stackrel{\mu}{\longrightarrow} & \widetilde{R}^{(p+1)}_{Y}.
\end{array}
$$
Here the vertical arrows are provided by Proposition \ref{CompThm1}
and the above isomorphism, and the lower horizontal arrow is the
ring multiplication map.
\end{proposition}
\begin{proof}
First we recall that in \S3 we associate each element in $S$ a
smooth complete intersection of four quadrics $Y$ in $\P^7$. Let
$g_0: \sZ \to Z$ be a good family over the moduli space of smooth
complete intersection of four quadrics in $\P^7$. So one has the
natural closed embedding $S\hookrightarrow Z$. Put $$h_0=g_0|_{S}:
\sY=\sZ|_{g_{0}^{-1}(S)}\to S.$$ Globalizing the construction in
Proposition \ref{Prop34} to $h_0$, one obtains a new family $\tilde
h_0: \widetilde \sY\to S$ admitting $N_1$-action over $S$, and by
taking quotient of it under $N_1$-action one recovers the family
$f_0: \sX\to S$. In summary, one has the following commutative
diagrams
$$
\begin{xy}
   \xymatrix{
      \sX\ar[d]_{f_0}& \ar[l]_{\pi}\widetilde \sY \ar[d]^{\tilde h_0}  \ar[r]^{\sigma} &  \sY \ar[d]^{h_0=g_0|_{S}}\ar[r]^{\hookrightarrow} & \sZ \ar[d]^{g_0} \\
        S &\ar[l]^{=} S \ar[r]_{=}& S\ar[r]^{\hookrightarrow}  & Z.
  }
\end{xy}
$$
Now we let $\W=R^3g_{0*}\Q_{\sZ}$ and $(F,\eta)$ be the system of
Hodge bundles associated with $\W$. By construction, it is clear
that $\V\simeq (\W|_{S})^{N_1}$ as PVHS. It follows that for $v\in
\sT_{S,s}\subset \sT_{Z,s}$ one has natural identification
$$
\theta_s(v)(\alpha)=\eta_{s}(v)(\alpha),
$$
where $\alpha\in \rm H^{3-p,p}(\widetilde X)\simeq \rm H^{3-p,p}(Y)^{N_1} \subset \rm H^{3-p,p}(Y)$.\\

Furthermore, for the IVHS of $\W$ at $s$, one has the following
commutative diagram (see Proposition 2.6 in \cite{Te})
$$
\begin{array}{ccc}
\sT_{Z,s}   & \stackrel{\eta_s}{\longrightarrow}
& \bigoplus_{p}\Hom(\rm{H}^{3-p,p}(Y),\rm{H}^{2-p,p+1}(Y)) \\[0.2cm]
{\scriptstyle\cong} \downarrow & & \downarrow {\scriptstyle\cong}
\\[0.2cm]
R^{(1)}_{Y}
 & \stackrel{}{\longrightarrow} & \bigoplus_{p}\Hom(R^{(p)}_{Y},R^{(p+1)}_{Y}),
\end{array}
$$
where the lower horizontal arrow is induced by the ring multiplication map. \\

Finally one has the following commutative diagrams by the
construction of the families $h_0, \tilde h_0$ and $f_0$:
$$
\begin{xy}
   \xymatrix{
      \sT_{S,s}\ar[d]_{\cap}\ar[r]^{\rho_{f_0,s}}& \rm H^1(\widetilde X,\sT_{\widetilde X}) \ar[d]^{ }  \ar[r]^{\simeq} &  \rm H^{2,1}(\widetilde X)\ar[d]^{ }  \\ \sT_{Z,s} \ar[r]^{\rho_{g_0,s}}  &\rm H^1(Y,\sT_{Y})\ar[r]^{\simeq}& \rm H^{2,1}(Y).
  }
\end{xy}
$$
This shows that under the left vertical isomorphism $\sT_{Z,s}\to
R^{(1)}_{Y}$ (in the second paragraph), the image of $\sT_{S,s}$ is
exactly $\widetilde{R}^{(1)}_{Y}$. The proposition follows by
putting everything above together.
\end{proof}

\subsection{Calculation}
Now we proceed to describe the computation of the characteristic
subvarieties introduced in \S \ref{Plethysm}. According to
Proposition \ref{PropGM}, the action of Higgs field on the
cohomology classes along a given tangent vector is equivalent to the
multiplication of corresponding elements in $\widetilde{R}_{Y}$ with
some fixed element in $\widetilde{R}^{(1)}_{Y}$. Furthermore, since
$\widetilde{R}_{Y}^{(0)}$ is one-dimensional, there exists a
isomorphism $\Hom(\widetilde{R}_{Y}^{(0)},\widetilde{R}_{Y}^{(k)})
\cong \widetilde{R}_{Y}^{(k)}$. Hence in our case the $k$-th
iterated Higgs fields
$$
(\theta_s)^k : \Sym^k(T_{S,s}) \longrightarrow \Hom(
\rm{H}^{3,0}(\widetilde{X}), \rm{H}^{3-k,k}(\widetilde{X}) )
$$
is given by the multiplication map
$$
\mu_k : \Sym^k(\widetilde{R}^{(1)}_{Y}) \longrightarrow
\widetilde{R}^{(k)}_{Y}.
$$
It follows that the $k$-th characteristic subvariety at $s$ is
isomorphic to
$$
(\sC_k)_{s} \quad = \quad
\P(\Sym^\cdot((\widetilde{R}_{Y}^{(1)})^*/\mathfrak{a}_{k+1}) \quad,
$$
where $\mathfrak{a}_{k+1}$ denotes the image of the dual
$\mu_{k+1}^* : (\widetilde{R}^{(k+1)}_{Y})^* \rightarrow
\Sym^{k+1}((\widetilde{R}^{(1)}_{Y})^*)$ of the multiplication map.
Fixing a basis of $\widetilde{R}^{(1)}_{Y}$ determines an
isomorphism $\P(\Sym^\cdot((\widetilde{R}_{Y}^{(1)})^*)) \cong
\P^8$. This shows that the first characteristic subvariety
$(\sC_1)_{s}$ can be computed by the following steps. Let
$(u_1,...,u_9)$ and $(v_1,...,v_9)$ denote the elements of $\sB_1$
and $\sB_2$, respectively.

\begin{enumerate}
\item[(1)] Fix a bijection
$$
\phi : \{ (i,j) \in \N^2 \mid 1 \leq i \leq j \leq 9 \}
\stackrel{\sim}{\longrightarrow} \{ 1,...,45 \} \quad,
$$
and define a basis $\sB = (w_1,...,w_{45})$ of
$\Sym^2(\widetilde{R}_{Y}^{(1)})$ by $w_{\phi(i,j)} = u_i u_j$.
Compute the representation matrix $C \in M(9\times 45,\C)$ of the
multiplication map $\mu_2$ with respect to
$\sB$ and $\sB_2$.\\[-0.2cm]
\item[(2)] Let $\sB_1^* = (u_1^*,...,u_9^*)$ denote the dual basis of
$\sB_1$, and let the basis $\widetilde{\sB} = (\tilde{w}_1,
...,\tilde{w}_{45})$ of $\Sym^2((\widetilde{R}^{(1)}_{Y})^*)$ be
defined by $\tilde{w}_{\phi(i,j)} = u_i^* u_j^*$. Determine the
representation matrix $D = (d_{ij}) \in M(45\times 9,\C)$ of the
dualized multiplication $\mu_2^*$
with respect to $\sB_2^*$ and $\widetilde{\sB}$.\\[-0.2cm]
\item[(3)] For the $k$-th column of $D$ define a polynomial
$f_k \in \C[z_1,...,z_9]$ by
$$
f_k \quad = \quad \sum_{i=1}^9 \sum_{j=i}^9 d_{\phi(i,j)k} z_i z_j.
$$
Let $\mathfrak{a}_2$ be the ideal generated by $f_1,...,f_9$. Then
$(\sC_1)_{s}$ is isomorphic to the projective subvariety in $\P^8$
which corresponds to $\mathfrak{a}_2$.
\end{enumerate}

\begin{remark}
Step (1) can be carried out in practice using computer algebra. For
each pair $(i,j)$, one computes the product $g_1 g_2$ of the
polynomials $g_i,g_j \in R_{8,4}$ which correspond to the basis
vectors $w_i,w_j \in \sB_1$ and reduces it with respect to the
Jacobian ideal $\mathfrak{J}_{\mathfrak{A}}$. The result can be
expressed as a linear combination of the elements in $\sB_2$. For
step (2) notice that $\widetilde{\sB}$ is related to the dual basis
$\widetilde{\sB}^*$ of $\sB$ by
$$
\tilde{w}_{\phi(i,j)} \quad = \quad
\begin{cases} \tilde{w}_{\phi(i,j)}^* & \text{if} \quad i = j \\
\tfrac{1}{2}\tilde{w}_{\phi(i,j)}^* & \text{if} \quad i \neq j.
\end{cases}
$$
By a similar procedure, we can compute the characteristic subvariety
$(\sC_2)_{s}$. We skip the details here, because it will not be used
in the sequel.
\end{remark}

\subsection{A general point and a special point}
Here we summarize the results of the calculations that can be
performed by the method explained above. First, if we pick a
generic configuration $\mathfrak{A}$ giving a general point
$\eta_0 \in \mathfrak{M}_{CY}$ and determine the ideal
$$ \mathfrak{a}_2 \subset \C[z_1,\ldots,z_9] $$
consisting of $9$ quadratic polynomials. Using a Gr\"obner-basis
calculation one can verify the following (see \cite{computation} for
details):

\begin{proposition}\label{Char1Result generic}
The characteristic subvariety $(\sC_1)_{\eta_0}$ over a general
point $\eta_0$ in the moduli space $\mathfrak{M}_{CY}$ is empty set.
\end{proposition}

We will now chooses a special point $s_0 \in
\widetilde{\mathfrak{H}}_{CY}$ represented by the matrix $A \in
M(8\times 4,\C)$ with
$$
A^t \quad = \quad
\begin{pmatrix}
1 & 1 & 1 & 1 & 1 & 1 & 1 & 1\\
1 & 2 & 3 & 4 & 5 & 6 & 7 & 8\\
1 & 4 & 9 & 16 & 25 & 36 & 49 & 64 \\
1 & 8 & 27 & 64 & 125 & 216 & 343 & 512
\end{pmatrix}
$$

\begin{proposition}\label{Char1Result}
The characteristic subvariety $(\sC_1)_{s_0}^{red}$ consists of two
irreducible surfaces of degree $6$, each spanning the same
$\P^7 \subset \P^8$.
\end{proposition}

\begin{remark}
For the point $s_0$ the ideal $\mathfrak{a}_2$ 9 quadratic
polynomials in the variables $z_1,z_2,\ldots,z_9$. We analyzed this
ideal using the computer algebra system {\sc Singular}. Via a
Hilbert-series computation, one shows that the variety in $\P^8$ has
dimension two an is of degree $12$. In order to find the irreducible
components of this surface, one has to make a primary decomposition
of the ideal $\mathfrak{a}_2$. For this we had to use several
tricks.\\

By elimination of the first five variables we obtain an hypersurface
of degree 8 in variables $z_6,z_7,z_8,z_9$. By looking at reductions
modulo various primes $p$, it was observed that the octic factors as
product of two quartics over $\F_p$ in case $p = 1 \mod 4$. With
some more work one finds a factorization of the octic into two
quartics over the field $\Q(\sqrt{-1})$.\\

Both quartics have a smooth twisted cubic as singular locus. Because
the generic plane section is a three nodal quartic, we see that
these two quartic surfaces are {\em irreducible} over $\C$. The
decomposition of the degree 8 surface into two quartics gives a
splitting of the degree 12 surface in $\P^8$ into two components,
which are surfaces of degree $6$.\\

The change in degree from $12$ to $8$ is due to the fact the
projecting out the first five variables is {\em non-generic}. If
instead we eliminate the variables $z_1,z_5,z_6,z_8,z_9$ one finds
that the degree $6$ components project to sextic surfaces, whose
plane section is a $10$-nodal sextic, hence rational. Such surfaces
are in fact ruled and can be obtained as join in $\P^7$ of
corresponding points on a conic and rational normal curve of degree
four and are cut out by 15 quadrics in $\P^7$. Indeed, it turns out
that both components are contained in the hyperplane of $\P^8$ given
by
\[2269z_1-378z_2+21x_3-1029z_4+147z_5-7z_6+192z_7-24z_8+z_9=0\]
By taking ideal quotients one can find a complete primary
decomposition of the ideal $\mathfrak{a}_2$. There are embedded components of
dimension one contained in the union of the two scrolls.
\end{remark}

\section{Proof of the Main Theorems}
\label{MainTheorems} \noindent

\begin{theorem}\label{MainTheorem1}
Let $f: \sX\to S$ be a good family of $\mathfrak{M}_{CY}$ and $\V$
be the associated weight 3 PVHS. Then $\V$ does not factor
canonically.
\end{theorem}
\begin{proof}
Assume the contrary. By Corollary \ref{CharSubMainCorollary}, for
any $s\in S$ away from the ramification locus of the moduli map
$S\to \mathfrak{M}_{CY}$ the first characteristic subvariety
$(\sC_{1})_{s}$ is then isomorphic to either $\P^2\times \P^2$ or
$P\cup Q$. In particular, in both cases there exists an irreducible
component in $(\sC_{1})_{s}$ whose dimension is greater than 2. By
Proposition \ref{Char1Result generic} the fiber of the first
characteristic subvariety at a general point $[\eta_0]$, the image
of $\eta_0$ in $\mathfrak{M}_{CY}$, is empty. By semi-continuity,
there is an open neighborhood of $[\eta_0]$ in $\mathfrak{M}_{CY}$
such that the fibers of the first characteristic subvariety over the
closed points in it are empty. Therefore there exists also a closed
point in $S$ away from the ramification locus of the moduli map of
$f$, over which the fiber of the first characteristic subvariety is
empty. This gives a contradiction.
\end{proof}

Our next aim is to prove Theorem \ref{MainTheorem2}. For the proof
of it, we use the notations as in \S1. We first prove the following

\begin{lemma}\label{FactorLemma}
Each factor $\V_i$ has quasi-unipotent local monodromy around each
irreducible component of $Z$.
\end{lemma}
\begin{proof} We put $\V'=\V_2\otimes\cdots\otimes \V_k$ and $n=\rank \V'$.
One considers the $\C$-PVHS $\bigwedge^n\V=\bigwedge^n(\V_1\otimes
\V')$. By Exercise 6.11 (b) in \cite{FH} $\Sym^n\V_1$ is a direct
factor of $\bigwedge^n\V$. Since $\V$ is of quasi-unipotent local
monodromy, the same holds for each direct factor of
$\bigwedge^{n}\V$, in particular for $\Sym^{n}(\V_1)$. Thus one
induces that $\V_1$ itself is also of quasi-unipotent local
monodromy, and therefore so is $\V'$. By induction on the number of
factors $k$ in the tensor decomposition of $\V$, one concludes that
each factor has quasi-unipotent local monodromy.
\end{proof}

\begin{theorem}\label{MainTheorem2}
Each local system $\V_i$ admits the structure of a $\C$-PVHS such
that the induced $\C$-PVHS on the tensor product $\V_1
\otimes\cdots\otimes \V_k$ coincides with the given $\C$-PVHS on
$\V$.
\end{theorem}
\begin{proof}
Let $s\in S$ be a base point and let $\rho_i : \pi_1(S,s)
\rightarrow G_i$ be the monodromy representation of $\V_i$. We put
$V_1$ to be the fiber of $\V_1$ at $s$. Since $\rho_i$ is a Zariski
dense representation into the simple Lie group $G_i$ with
quasi-unipotent local monodromy around $Z$, by \cite{JZ} there
exists a pluri-harmonic metric on the flat bundle $\V_i$ with finite
energy, which makes $\V_i$ into a Higgs bundle $(E_i,\theta_i)$ over
$S$. Furthermore in \cite{M} Mochizuki has analyzed the singularity
of this harmonic metric in detail and has shown that
$(E_i,\theta_i)$ admits a logarithmic extension
$(\bar{E}_i,\bar{\theta}_i)$ over $\bar S$, i.e.\ a vector bundle
$\bar{E}_i$ over $\bar S$ which extends $E_i$ and a map
$$
\bar{\theta} : \bar{E}_i \longrightarrow \bar{E}_i \otimes
\Omega^1_{\bar S}(\log Z)
$$
which coincides with $\theta$ over $S$. Such a pluri-harmonic metric
is called tame. In this case the residue of $\bar{\theta}$ along $Z$
is unipotent.\\

From the proof of Lemma \ref{FactorLemma} we know that one finds a
non-trivial component $\Sym^n(\V_1)$ in $\bigwedge^n(\V_1 \otimes
\V')$. Since $\rho_1$ is Zariski dense in $G_1$ and $G_1$ is simple,
one finds a suitable Schur functor $\SSS_\mu$ such that
$\SSS_\mu(\rho_1)$ is a nontrivial irreducible direct factor of
$\bigwedge^n(\V_1 \otimes \V')$. Since $\bigwedge^n(\V_1 \otimes
\V')$ is semi-simple, there exists a decomposition
$$
\bigwedge^n(\V_1 \otimes \V') \quad = \quad \bigoplus_{i=1}^m
\W_{i1} \otimes \W_{i2}
$$
where the $\W_{i1}$ are irreducible and the $\W_{i2}$ are trivial.
By Proposition 1.13 in \cite{D1}, there exist uniquely determined
$\C$-PVHS on the $\W_{i1}$ and complex Hodge structures on the
$\W_{i2}$ such that the direct sum of the tensor products of them
coincides with the $\C$-PVHS on $\bigwedge^n(\V_1 \otimes \V')$.
In particular, there exists a $\C$-PVHS on $\SSS_\mu(\V_1)$.\\

By the uniqueness of the such pluri-harmonic metrics, the induced
pluri-harmonic metric on $\SSS_\mu(\bar{E}_1,\bar \theta_1)$
coincides with that of the $\C$-PVHS on $\SSS_\mu(\V_1)$. Hence
$\SSS_\mu(\bar{E}_1,\bar \theta_1)$ is a fixed point of the
$\C^\times$-action. The Schur functor $\SSS_\mu$ induces a
nontrivial morphism $G_1\to GL(\SSS_\mu(V_1))$, which is injective
since $G_1$ is simple. It induces the morphism
$$
\phi_\mu : \mathfrak{M}(\pi_1(S),G_1)^{\rm{ss}} \longrightarrow
\mathfrak{M}(\pi_1(S),\GL(\SSS_\mu(V_1)))^{\rm{ss}}
$$
between the corresponding moduli spaces of semi-simple
representations. By Corollary 9.16 in \cite{Sim1}, the morphism
$\phi_\mu$ is finite.\\

If $Z = \emptyset$, then $\C^\times$ acts on both moduli spaces
continuously via the Hermitian Yang-Mills metric on poly-stable
Higgs bundles $(E,t\theta)$, and this action is compatible with
$\phi_\mu$. Since $\SSS_\mu(\rho_1)$ is a fixed point of the
$\C^\times$-action, the representation $\rho_1$ itself is a fixed
point of the $\C^\times$-action. Hence $(E_1,\theta_1)$ is a
$\C$-PVHS on $\V_1$. Now we consider the general situation $Z \neq
\emptyset$. Let $C \subseteq \bar S$ denote a curve which is a
complete intersection of ample hypersurfaces, and define $C_0 = C
\setminus Z$. Taking the restriction
$$
\rho_1|_{C_0} \in \mathfrak{M}(\pi_1(C_0),G_1)^{\rm{ss}} \quad,
$$
we have $\SSS_\mu(\rho_1)|_{C_0} \in
\mathfrak{M}(\pi_1(C_0),\GL(\SSS_\mu(V_1)))^{\rm{ss}}$. Now we
consider the map
$$
\phi_{\mu} : \mathfrak{M}(\pi_1(C_0),G_1)^{\rm{ss}} \longrightarrow
\mathfrak{M}(\pi_1(C_0),\GL(\SSS_\mu(V_1)))^{\rm{ss}}.
$$
By Simpson's main theorem in \cite{Sim0}, there exist Hermitian
Yang-Mills metrics on poly-stable Higgs bundles on $C$ with
logarithmic poles of the Higgs field on $C\cap Z$. The
$\C^\times$-action can be defined on both spaces of semi-simple
representations on $C_0$ via a Hermitian Yang-Mills metric on
$(\bar{E},t\bar{\theta})$. Applying the same arguments as above to
the compact case, we show that the pullback of
$(\bar{E}_1,\bar{\theta}_1)$ to $C_0$ is a fixed point of the
$\C^\times$-action. If we choose $C_0$ sufficiently ample, then
$(\bar{E}_1,\bar{\theta}_1)$ is also a fixed point of the
$\C^\times$-action. Again by \cite{Sim0},
$(\bar{E}_1,\bar{\theta}_1)$ is a $\C$-PVHS on $\V_1$.\\

Since $\V'$ is a direct factor $\V_1^*\otimes \V_1\otimes
\V'=\V_1^*\otimes \V$, $\V'$ admits a $\C$-PVHS as well. The tensor
product of $\C$-PVHS on $\V_1$ and $\V'$ is a $\C$-PVHS on $\V_1
\otimes \V'$. By Deligne's uniqueness theorem on $\C$-PVHS on
irreducible local systems, the tensor product coincides with the
original $\C$-PVHS on $\V_1 \otimes \V'=\V$. By induction on the
number of factors $k$, we conclude the proof of the theorem.
\end{proof}

\begin{theorem}\label{Classification of Zariski closure}
Let $S$ be a smooth quasi-projective algebraic variety and $\V$ be a
weight 3 $\Z$-PVHS over $S$ which is irreducible over $\C$ and of
quasi-unipotent local monodromies. If the Hodge numbers of $\V$ are
$(1,9,9,1)$, then after a possible finite \'{e}tale base change the
connected component of the real Zariski closure of the monodromy
group of $\V$ is one of the following:
\begin{itemize}
  \item [(A)] $SU(1,1)\times SO_{0}(2,8)$,
  \item [(B)] $SU(3,3)$,
  \item [(C)] $Sp(6,\R)$,
  \item [(D)] $Sp(20,\R)$.
\end{itemize}
\end{theorem}
\begin{proof}
Let $s\in S$ be a base point of $S$ and $V$ be the fiber of $\V$ at
$s$. Let $\tau: \pi_1(S,s) \longrightarrow GL(V\otimes_{\Z}\R)$ be
the monodromy representation of $\V\otimes \R$ and $G$ be the
Zariski closure of $\tau$. So we have the factorization
$$
\tau: \pi_1(S,s) \to G\stackrel{\rho}{\longrightarrow} GL(V_{\R})
$$
where $\rho$ is a morphism of real algebraic groups. Since $\V$ is
of polarized and of $\Z$-coefficients, $G$ is semi-simple by Deligne
(cf. Corollary 4.2.9 in \cite{D0}). Since $\V$ is of weight 3 and
the dimension of $V$ is 20, $G$ is a semi-simple real Lie subgroup
of $Sp(20,\R)$. Let $G^0$ be the connected component of $G$. The
classification of $G^0$ consists of several steps.\\

{\bf Step 1.} Let $\mathfrak g$ be the Lie algebra of $G^0$ and
$\chi: \mathfrak{g}_\C \longrightarrow \mathfrak{sp}_{20}\C$ be the
complexification of the differential of $\rho$. Then the pair
$(\mathfrak{g}_\C,\chi)$ is one of the following list:
$$
\begin{array}{llcll}
\text{(1)} & (\mathfrak{sl}(2),\Gamma_{19}) & \hspace*{1cm} &
\text{(7)} & (\mathfrak{sl}(2) \oplus \mathfrak{so}(5),\Gamma_3
\otimes \Gamma_{10})\\
\text{(2)} & (\mathfrak{so}(5),\Gamma_{03}) & \hspace*{1cm} &
\text{(8)} & (\mathfrak{so}(5) \oplus \mathfrak{sl}(2), \Gamma_{01} \otimes \Gamma_4) \\
\text{(3)} & (\mathfrak{sp}(6),\Gamma_{00100}) & \hspace*{1cm} &
\text{(9)} &(\mathfrak{so}(5) \oplus
\mathfrak{so}(5), \Gamma_{10} \otimes \Gamma_{02}) \\
\text{(4)} & (\mathfrak{sl}(6),\Gamma_{00100}) & \hspace*{1cm} &
\text{(10)} &(\mathfrak{sl}(2) \oplus \mathfrak{so}(5),\Gamma_1 \otimes \Gamma_{02}) \\
\text{(5)} & (\mathfrak{sp}(20),\Gamma_{1000000000}) & \hspace*{1cm}
&
\text{(11)} & (\mathfrak{sl}(2) \oplus \mathfrak{so}(10), \Gamma_1 \otimes \Gamma_{10000}) \\
\text{(6)} &(\mathfrak{sl}(2) \oplus \mathfrak{sl}(2), \Gamma_3
\otimes \Gamma_4) \\
\end{array}
$$
Here we use the notations as given in \cite{FH}. The list results
from a rather standard calculation in the representation theory of
semi-simple complex Lie algebras. As $\mathfrak{g}_\C$ is
semi-simple, we can write $\mathfrak{g}_\C = \oplus_{i=1}^m
\mathfrak{g}_i$ into direct sum of simple Lie algebras. By Schur's
lemma and since $\chi$ is irreducible, we have the tensor
decomposition of $\chi=\otimes \chi_i$ into irreducible
representations. Then one has particularly $\prod_{i=1}^m d_i = 20$,
where $d_i$ is the dimension of the representation space of
$\chi_i$. It is straightforward to write down a complete list of
irreducible representations of complex simple Lie algebras whose
dimensions divide 20. For each pair $(\mathfrak{g}_\C,\chi)$ where
$\chi: \mathfrak{g}_\C \to \mathfrak{sl}(20)$ in this preliminary
list, it appears in the final list, i.e. $\chi$ factors through
$\mathfrak{sp}(20)\subset \mathfrak{sl}(20)$ if and only if
$\wedge^2\chi$ contains a trivial representation. This can be easily
checked by using the plethysm of semi-simple complex Lie algebras.\\

{\bf Step 2.} Let $G_{\C}$ be the complexification of $G^0$. Then by
Lemma 4.4 in \cite{Sim}, $G^0$ is a real form of $G_{\C}$ which is
of Hodge type (see \S4 \cite{Sim} for the definition). We shall make
use of the list of all simple real Lie groups of Hodge type on Page
50 \cite{Sim}. In connection with the list in Step 1, one can
immediately exclude the case where $\mathfrak g$ is simple but
$\mathfrak{g}_\C$ is not simple by Proposition 4.4.10 in \cite{Sim}.
Hence the number of irreducible factors in $\mathfrak g$ is equal to
that of $\mathfrak g_\C$. \\

{\bf Step 3.} We start with two factors. Namely, $G^0=G_1\times G_2$
with $G_i$ simple real Lie groups. It induces the tensor
decomposition of real local systems $\V_{\R}=\V_1\otimes \V_2$. By
Theorem \ref{MainTheorem2} there exist $\C$-PVHS structures on
$\V_1$ and $\V_2$ such that the induced PVHS on $\V$ coincides with
the original one. It follows that the Lie groups $G_1$ and $G_2$ are
also of Hodge type. By Lemma 5.5 in \cite{Sim} $\V_i,\ i=1,2$
underlies $\R$-PVHS structure. Recall that the weight of $\V$ is
three and its Hodge numbers are $1,9,9,1$. Then after a possible
permutation of factors the PVHS $\V_1$ must be of weight $1$ with
Hodge numbers $1,1$ and $\V_2$ is of weight $2$ with Hodge numbers
$1,8,1$. This implies $G_1=SU(1,1)$ and $G_2 \subseteq SO(2,8)$.
This excludes immediately the case (8) of the list in Step 1. It
excludes also the cases (4)-(7) by checking the dimension of
representation on $\mathfrak{sl}(2)$-factor. Let us now consider the
case (9). Note that the representation $\Gamma_{02}$ is simply the
second wedge power of the standard representation in this case. By
the list in \cite{Sim} $G_2$ can be one of the groups $SO(5)$,
$SO(1,4)$ or $SO(2,3)$. Since $SO(5)$ is compact, it is mapped into
the compact form $SO(10)$ of $SO(10,\C)$ under $\wedge^2$ of the
standard representation of $SO(5,\C)$. Also one checks that under
the same representation the other two real forms $SO(1,4)$ and
$SO(2,3)$ are mapped into the real forms $SO(6,4)$ and $SO(4,6)$
respectively. Thus the case (9) can be excluded. So it remains (10)
for the non-simple case.
Obviously $G_2=SO_0(2,8)$ and this gives the case (A). \\

{\bf Step 4.} We treat the case that $\mathfrak{g}_\C$ is simple. By
the list in \cite{Sim}, $G^0$ can be $SU(1,1)$ in case (1), $SO(5)$
,$SO(2,3)$ and $SO(1,4)$ in case (2), $SU(p,6-p)$ in case (3),
$Sp(20,\R)$ in case (4). Note that except for $SO(1,4)$ the rest
groups are of Hermitian type. For them, by consideration of weight
and Hodge numbers as in Step 3 we can exclude all cases except
$SU(3,3)$ and $Sp(20,\R)$, which give case (B) and (C) respectively.
Finally one can check directly that the real form $SO(1,4)$ of
$SO(5,\C)$ does not map into the split form $Sp(20,\R)$ of
$Sp(20,\C)$ under the third wedge power. For this, one can consult
for example Example 3, \S7 in \cite{O}. This completes the
classification.
\end{proof}

\begin{theorem}\label{MainTheorem3}
Let $f: \sX\to S$ be a good family of $\mathfrak{M}_{CY}$ and $\V$
be the associated weight 3 $\Z$-PVHS. Let $s \in S$ be a base point
and let
$$
\tau:\pi_1(S,s) \to Sp(20,\R)
$$
be the monodromy representation associated to $\V\otimes \R$. Then
the image of $\tau$ is Zariski dense in $Sp(20,\R)$.
\end{theorem}
\begin{proof}
Let $G$ be the real Zariski closure of $\tau$. Note that the Zariski
dense property is not changed under a finite \'{e}tale base change.
By Theorem \ref{Classification of Zariski closure} it is to show for
$\V$ in the statement the connected component $G^0$ of $G$ can not
be the case (A) or (B). Assume the contrary and we shall deduce a
contradiction in the following. Note that by the proof of Theorem
\ref{Classification of Zariski closure}, the inclusion $G\to
Sp(20,\R)$ comes up with a uniquely determined morphism of real
algebraic groups $\rho: G\to Sp(20,\R)$, which restricts to $\rho_0$
on $G^0$. Actually $\rho_0$ has already appeared in Proposition
\ref{ModularTypeA} implicitly. It is easy to verify that $\rho_0$ in
each case maps the maximal compact group $K^0$ of $G^0$ into the
compact subgroup $U(1)\times U(9)$ of $Sp(20,\R)$. Let $\Gamma$ be
the monodromy group of $\tau$ and $\Gamma_0=\rho^{-1}(\Gamma)\subset
G$. So we get a factorization
$$
\tau: \pi_1(S,s) \stackrel{\tau_0}{\longrightarrow} \Gamma_0
\stackrel{\rho}{\longrightarrow} \Gamma.
$$
This gives a factorization of the period map of $\V$
$$
\phi: S \stackrel{j}{\longrightarrow}  \Gamma_0\backslash G/K
\stackrel{\psi}{\longrightarrow} \Gamma\backslash
Sp(20,\R)/U(1)\times U(9)
$$
where $K$ is the maximal compact group of $G$. Since $S$ is
connected, the morphism $j$ factors though $\Gamma_0'\backslash
G^0/K^0\subset \Gamma_0\backslash G/K$ for $\Gamma_0'=\Gamma_0\cap
G^0$. Thus we arrive at the factorization in Definition
\ref{geometric pseudo-modularity}. Since $\rho_0: G^0\to Sp(20,\R)$
gives rise to the canonical PVHS, the factorization contradicts with
the assertion of Theorem \ref{MainTheorem1}. The proof is completed.
\end{proof}

\begin{corollary}
The special Mumford-Tate group of a general member in
$\mathfrak{M}_{CY}$ is $Sp(20,\Q)$.
\end{corollary}
\begin{proof}
Let $f:\sX\to S$ be a good family for $\mathfrak{M}_{CY}$. Let $\V$
be the associated weight 3 $\Q$-PVHS of $f$ and $\tau: \pi_1(S,s)\to
Sp(20,\Q)$ be the monodromy representation. By Deligne and Schoen
(see for example Lemma 2.4 \cite{VZ1}), the connected component of
the $\Q$-Zariski closure of the monodromy group is a normal subgroup
of the special Mumford-Tate group $Hg(\V)$ of $\V$, which is equal
to the special Mumford-Tate group of a general closed fiber of $f$.
By Theorem \ref{MainTheorem3} the $\Q$-Zariski closure of the
monodromy group of $\V$ has to be the whole symplectic group
$Sp(20,\Q)$, the corollary follows since the moduli map of $f$ is
dominant.
\end{proof}

\begin{theorem}\label{MainTheorem4}
Let $\mathfrak{H}_{CY}$ be the hyperelliptic locus of
$\mathfrak{M}_{CY}$ and $\mathfrak{H}$ be any subvariety of
$\mathfrak{M}_{CY}$ which strictly contains $\mathfrak{H}_{CY}$. Let
$f: \sX\to S$ be a good family of $\mathfrak{M}_{CY}$ whose moduli
map $S\to \mathfrak{M}_{CY}$ is dominant over $\mathfrak{H}$. Then
the restriction of $\V$ to the inverse image of $\mathfrak{H}$ does
not factor through $(D^{III}_{3},\wedge^3)$.
\end{theorem}

\begin{proof}
We first prove the statement for the good family $f_0$ over
$\mathfrak{M}_{AR}$. Assume that an extension $\mathfrak{H}$ for
$f_0$ as in the theorem does exist. We can assume $\dim
\mathfrak{H}=6$ without loss of generality. Let $\widetilde
{\mathfrak{H}}$ be the inverse image of $\mathfrak{H}$ in
$\mathfrak{M}_{AR}$ and $f_{0}'=f_{0}|_{\widetilde{ \mathfrak{H}}}$
the restriction of $f_0$ to $\widetilde {\mathfrak{H}}\subset
\mathfrak{M}_{AR}$. Let $\phi$ be the period map of $f_0$ and
$\phi'$ that of $f_{0}'$. By assumption one has the factorization
$$
\phi': \widetilde {\mathfrak{H}}\stackrel{j}{\longrightarrow}
\Gamma_0\backslash Sp(6,\R)/U(3)\stackrel{\wedge^3}{\longrightarrow}
\Gamma\backslash Sp(20,\R)/U(1)\times U(9).
$$
By Corollary \ref{etale of moduli spaces} and the local Torelli
theorem for CY manifolds, $\phi$ is \'{e}tale over its image. So is
the restriction $\phi'$. Because $\Gamma_0\backslash Sp(6,\R)/U(3)$
is of six dimensional, $j$ is then \'{e}tale. We derive the
contradiction at the special point $s_0\in \widetilde
{\mathfrak{H}}_{CY}\subset \widetilde {\mathfrak{H}}$. We denote by
$\sC_1$ the first characteristic subvariety of $f_{0}$ and $\sC_1'$
that of $f_{0}'$. By Proposition \ref{irreducible case} (ii), the
fiber $(\sC_1')_{s_0}$ as subvariety of $\P(\sT_{\widetilde
{\mathfrak{H}},s_0})$ is isomorphic to the $\P^2$ into $\P^5$ via
the Veronese embedding. In particular $(\sC_1')_{s_{0}}$ is reduced.
On the other hand, by Lemma \ref{key lemma} we know that
$$
(\sC_1')_{s_{0}}=((\sC_1)_{s_{0}}\cap \P(\sT_{\widetilde
{\mathfrak{H}},s_0}))^{red},
$$
where the scheme-theoretical intersection of the right hand side is
taken in the projective space $\P(\sT_{\mathfrak{M}_{AR},s_0})$. Now
by Proposition \ref{Char1Result}, $(\sC_1)_{s_{0}}$ is two
dimensional and has two irreducible components which are not
contained in any linear projective subspace of dimension $\le 6$.
It follows that
$((\sC_1)_{s_{0}}\cap \P(\sT_{\widetilde
{\mathfrak{H}},s_0}))^{red}$ is of dimension $\leq 1$. A
contradiction. Thus such an extension
$\mathfrak{H}$ for $f_0$ does not exist.\\

As a consequence we get the maximal property of $\mathfrak{H}_{CY}$
stated in Corollary \ref{Maximality of special MT group}. This fact
shows in turn the non-extension property for other good families in
the theorem. Let $f$ be such a good family. As above we can assume
$\dim \mathfrak{H}=6$. We put $S'$ to be inverse image of
$\mathfrak{H}$ under the moduli map $S\to \mathfrak{M}_{CY}$. The
factorization of the period map gives the morphism $j: S'\to
\Gamma_0\backslash Sp(6,\R)/U(3)$ and it induces an isomorphism
$\V|_{S'}\simeq j^*\W$ with $\W$ the canonical $\R$-PVHS over
$\Gamma_0\backslash Sp(6,\R)/U(3)$. Let $\V_{\Q}$ be the $\Q$-PVHS
associated to $f$. The canonical PVHS $\W$ over $\Gamma_0\backslash
Sp(6,\R)/U(3)$ has also a natural $\Q$-structure $\W_{\Q}$ such that
over the points in $\mathfrak{H}_{CY}$ the isomorphism between $\V$
and $j^*\W$ is defined over $\Q$ . Thus we have actually isomorphism
$(\V_{\Q})|_{S'}\simeq j^*\W_{\Q}$. This implies that the special
Mumford-Tate group of a general closed point in $S'$ and hence in
$\mathfrak H$ is contained in $Sp(6,\Q)$. This contradicts with
Corollary \ref{Maximality of special MT group}.
\end{proof}

\noindent {\bf Acknowledgements:} The major part of this work was
done during an academic visit of the second named author to the
Department of Mathematics of the University of Mainz in 2006. He
would like to express his hearty thanks to the hospitality of the
faculty, especially Stefan M\"{u}ller-Stach. We also thank Eckart
Viehweg for his interest and several helpful conversations about
this work.

\end{document}